\documentclass[english, 10pt]{amsart}

\usepackage{babel}
\usepackage{amstext}
\usepackage{amsmath}
\usepackage{amssymb}
\usepackage{amsfonts}
\usepackage{latexsym} 
\usepackage{ifthen}
\usepackage{xypic}
\xyoption{all}
\usepackage{verbatim}

\newtheorem{lemma1}{}[section]

\newenvironment{lemma}{\begin{lemma1}{\bf Lemma.}}{\end{lemma1}}
\newenvironment{example}{\begin{lemma1}{\bf Example.}\rm}{\end{lemma1}}

\newenvironment{theorem}{\begin{lemma1}{\bf Theorem.}}{\end{lemma1}}
\newenvironment{proposition}{\begin{lemma1}{\bf Proposition.}}{\end{lemma1}}

\newenvironment{remark}{\begin{lemma1}{\bf Remark.}\rm}{\end{lemma1}}
\newenvironment{definition}{\begin{lemma1}{\bf Definition.}}{\end{lemma1}}

\newenvironment{conjecture}{\begin {lemma1}{\bf Conjecture.}}{\end{lemma1}}

\newenvironment{remark*}{{\bf Remark.}}{}
\newenvironment{example*}{{\bf Example.}}{}

\newcommand{\Q}{\ensuremath{\mathbb{Q}}}
\newcommand{\Z}{\ensuremath{\mathbb{Z}}}
\newcommand{\C}{\ensuremath{\mathbb{C}}}
\newcommand{\N}{\ensuremath{\mathbb{N}}}
\newcommand{\PP}{\ensuremath{\mathbb{P}}}

\newcommand{\holom}[3]{\ensuremath{#1\colon #2  \rightarrow #3}}
\newcommand{\fibre}[2]{\ensuremath{#1^{-1} (#2)}}

\makeatletter
\ifnum\@ptsize=0  \fi \ifnum\@ptsize=2  \fi 

\newcommand\sO{{\mathcal O}}

\DeclareMathOperator*{\red}{red}

\DeclareMathOperator*{\nons}{nons}

\newcommand\pic{\ensuremath{\mbox{Pic}}}

\newcommand{\chow}[1]{\ensuremath{\mathcal{C}(#1)}}

\DeclareMathOperator*{\Rat}{RatCurves^n}
\newcommand{\Hilb}{\ensuremath{\mathcal{H}}}
\newcommand{\Univ}{\ensuremath{\mathcal{U}}}

\title{Mori contractions of maximal length} 
\date{\today}

\author{Andreas H\"oring}
\author{Carla Novelli}

\thanks{A.H. was partially supported by the A.N.R. project ``CLASS''}

\address{Andreas H\"oring, Universit{\'e} Pierre et Marie Curie, Institut de math{\'e}matiques de Jussieu,
Projet Topologie et g{\'e}om{\'e}trie alg{\'e}briques, Case 247,  4 place Jussieu, 75005 Paris, France}
\email{hoering@math.jussieu.fr}

\address{Carla Novelli, Dipartimento di Matematica, Universit\`a degli Studi di Padova, via Trie\-ste 63, I-35121 Padova, Italy}
\email{novelli@math.unipd.it}

\begin{document}

\begin{abstract}
We prove a relative version of the theorem of Cho, Miyaoka and Shepherd-Barron$\colon$
a Mori fibre space of maximal length is birational to a projective bundle.
\end{abstract}

\maketitle

\section{Introduction}

\subsection{Motivation}
Let $X$ be a Fano manifold of dimension $d$ with Picard number one, and denote by $H$ the ample
generator of the Picard group. Let $i(X) \in \N$ be such that $K_X^* \equiv i(X) H$; by a classical
theorem of Kobayashi and Ochiai \cite{KO73} one has $i(X) \leq d+1$ and equality holds if and only
if $X$ is isomorphic to the projective space $\PP^d$. If one tries to understand Fano manifolds
of higher Picard number the index $i(X)$ is less useful, since it can be equal to one even for 
very simple manifolds like $\PP^d \times \PP^{d+1}$. For these manifolds the pseudoindex
$$
\min \{
K_X^* \cdot C \ | \ C \subset X \ \mbox{a rational curve}
\}
$$
yields much more precise results, the most well-known being the theorem of Cho, Miyaoka
and Shepherd-Barron:
\begin{theorem} \cite[Cor.0.3]{CMS02} \cite[Thm.1.1]{Keb02} \label{theoremCMS}
Let $X$ be a projective manifold of dimension $d$ such that for every rational curve $C \subset X$,
we have $K_X^* \cdot C \geq d+1$. Then $X$ is isomorphic to the projective space $\PP^d$.
\end{theorem}

Fano manifolds (or more generally Fano varieties with certain singularities) are important objects
since they appear naturally in the minimal model program as the general fibres of Mori fibre spaces. 
However if one wants to get a more complete picture of Mori fibre spaces, one should also try to obtain some information on the special fibres. In the polarised setting we have a relative version
of the Kobayashi--Ochiai theorem:

\begin{theorem} \cite[Lemma 2.12]{Fuj87}, \cite{Ion86} \label{theoremfujita}
Let $X$ be a manifold and $\holom{\varphi}{X}{Y}$ a projective, equidimensional morphism of relative dimension $d$ onto a normal variety.
Suppose that the general fibre $F$ is isomorphic to  $\PP^d$ and that there exists a $\varphi$-ample Cartier divisor $A$ on $X$ such that the restriction to $F$ is isomorphic to the hyperplane divisor $H$. Then $\varphi$ is a projective bundle and $A$ is a global hyperplane divisor.
\end{theorem}

The existence of the $\varphi$-ample Cartier divisor $A$ being a rather restrictive condition, 
the object of this paper is to replace it by a more flexible ``numerical'' hypothesis, i.e. we prove
a relative version of Cho--Miyaoka--Shepherd-Barron theorem, even for fibrations that are not equidimensional.

\subsection{Main results}

Let $X$ be a quasi-projective manifold and let \holom{\varphi}{X}{Y} be an {\em elementary contraction}, i.e. a Mori contraction associated with an extremal ray $R$ of $X$. 
The {\em length} of the extremal ray $R$ (or length of the elementary contraction) is defined as
$$
l(R) := \min \{
K_X^* \cdot C \ | \ C \subset X \ \mbox{a rational curve s.t.} \ [C] \in R
\}.
$$
Denote by $E \subset X$ an irreducible component of the $\varphi$-exceptional locus  ($E=X$ for a contraction of fibre type),
and let $F$ be an irreducible component of a $\varphi$-fibre contained in $E$. 
Then by the Ionescu--Wi\'sniewski inequality \cite[Thm.0.4]{Ion86}, \cite[Thm.1.1]{Wis91} one has
\begin{equation} \label{eqniw}
\dim E + \dim F \geq \dim X + l(R) -1.
\end{equation}

In this paper we investigate the case when \eqref{eqniw} is an equality:
if the contraction $\varphi$ maps $X$ onto a point, Theorem \ref{theoremCMS} 
says that $X=E=F$ is a projective space; more generally, Andreatta and Wi\'sniewski conjectured that $F$ is a projective space regardless of the dimension of the target
\cite[Conj.2.6]{AW97}. We prove a strong version of this conjecture for Mori fibre spaces:

\begin{theorem} \label{theoremmain}
Let $X$ be a quasi-projective manifold that admits an elementary contraction of fibre type
\holom{\varphi}{X}{Y} onto a normal variety $Y$ such that the general fibre has dimension $d$. 
Suppose that the contraction has length $l(R) = d+1$.

If $\varphi$ is equidimensional, it is a projective bundle. 
If $\varphi$ is not equidimensional, 
there exists a commutative diagram
$$
\xymatrix{
X'   \ar[d]_{\varphi'} \ar[r]^{\mu'} & X  \ar[d]^{\varphi} 
\\
Y' \ar[r]^{\mu}  & Y
}
$$
such that $\mu$ and $\mu'$ are birational, $X'$ and $Y'$ are smooth, and $\holom{\varphi'}{X'}{Y'}$ is 
a projective bundle.
\end{theorem}

Since $X$ is smooth, Theorem \ref{theoremCMS} immediately implies that a general $\varphi$-fibre $F$ is isomorphic to $\PP^{d}$. Our contribution is to
prove that the condition on the length severely limits the possible degenerations of these projective spaces. 
The smoothness of $X$ is not essential for these degeneration results (cf. Section \ref{sectiondegeneration}),
so one can easily derive analogues of Theorem \ref{theoremmain} making some assumption
on the singularities of the general fibre (e.g. isolated LCIQ singularities  \cite{CT07}).

If the contraction $\varphi$ is birational, the situation is more complicated: if $E \subset X$ is an irreducible component of the exceptional locus
such that for a general fibre $F$ of $E \rightarrow \varphi(E)=:Z$ the inequality \eqref{eqniw} is an equality,  it is not hard to see
(\cite[Rem.12]{CMS02}, cf. Remark \ref{remarkbirationalcase}) that
the general fibre is normalised by a finite union of projective spaces. Let $\tilde E \rightarrow E$ be the normalisation and
$\holom{\tilde \varphi}{\tilde E}{\tilde Z}$ be the fibration obtained by the Stein factorisation of $\tilde E \rightarrow E \rightarrow Z$.
The general $\tilde \varphi$-fibre is a projective space and we obtain the following:

\begin{theorem} \label{theorembirational}
In the situation above, the fibration $\holom{\tilde \varphi}{\tilde E}{\tilde Z}$ is a projective bundle in codimension one. Moreover 
there exists a commutative diagram
$$
\xymatrix{
E'   \ar[d]_{\varphi'} \ar[r]^{\mu'} & \tilde E  \ar[d]^{\tilde \varphi} 
\\
Z' \ar[r]^{\mu}  & \tilde Z
}
$$
such that $\mu$ and $\mu'$ are birational, $E'$ and $Z'$ are smooth, and $\holom{\varphi'}{E'}{Z'}$ is 
a projective bundle.
\end{theorem}

While this result gives a rather precise description of the normalisation $\tilde E$, it does not prove that the fibre $F$ itself
is a projective space. However, if the contraction $\varphi$ is divisorial (so the inequality \eqref{eqniw} simplifies
to $\dim F \geq l(R)$), Andreatta and Occhetta \cite[Thm.5.1]{AO02} proved that all the nontrivial fibres of $\varphi$ have dimension $l(R)$ if and only if $X$ is the blow-up of a manifold along a submanifolds of codimension $l(R)+1$.
If one studies the proof of Kawamata's classification of smooth fourfold flips \cite[Thm.1.1]{Ka89} (which corresponds to the case $\dim X=4, \dim E=\dim F=2, l(R)=1$) one sees that for flipping contractions the proof of the normality requires completely different techniques. We leave this interesting problem for future research.

\subsection{Further developments}
Theorem \ref{theoremmain} completely determines the equidimensional fibre type contractions of maximal length: they are projective bundles.
If the fibration is not equidimensional it still gives a precise description of the Chow family defined by the fibration. 
It seems reasonable that this description allows to deduce some information about higher-dimensional fibres. For example if $X$ of dimension $n$ maps
onto a threefold $Y$, it is not hard to see that any fibre component of dimension $n-2$ is normalised by $\PP^{n-2}$.
More generally we expect that the  theory of varieties covered by high-dimensional linear spaces \cite{Ein85}, \cite{Wis91b}, \cite{BSW92}, \cite{ABW92}, \cite{Sa97}, \cite{NO11} can be applied in this context. 

Another interesting line of investigation would be to prove that under additional assumptions the fibration $\varphi$ is always equidimensional. We recall the following conjecture by Beltrametti and Sommese:

\begin{conjecture} \cite{BS93}, \cite[Conj.14.1.10]{BS95} \label{conjecturescroll}
Let $(X,L)$ be a  polarised projective manifold of dimension $n$ that is an adjunction theoretic scroll
\holom{\varphi}{X}{Y} over a normal variety $Y$ of dimension $m$. If $n \geq 2m-1$, then $\varphi$ is equidimensional.
\end{conjecture}

Wi\'sniewski \cite[Thm.2.6]{Wis91b} proved this conjecture if $L$ is very ample and $n \geq 2m$, but apart from
partial results for low-dimensional $Y$, \cite{BSW92}, \cite{Somm85}, \cite{Ti10}, this conjecture is very much open.

In Section \ref{sectionscroll} we use Theorem \ref{theoremmain} to provide some evidence that a contraction of length $l(R) = n-m+1$ is always locally a scroll, so we expect that Conjecture \ref{conjecturescroll} even holds in the more general setting of Theorem \ref{theoremmain}.

\begin{conjecture} 
Let $X$ be a projective manifold of dimension $n$ that admits an elementary contraction of fibre type
\holom{\varphi}{X}{Y} onto a normal variety $Y$ of dimension $m$. Suppose that the contraction has length $l(R) = n-m+1$.
If $n \geq 2m-1$, then $\varphi$ is equidimensional.
\end{conjecture}

{\bf Acknowledgements.} We thank C. Araujo for sending us her inspiring preprint \cite{Ara09}.
We thank M. Andreatta, S. Druel, H. Hamm, F. Han, P. Popescu-Pampu, M. Wolff and J.A. Wi\'sniewski
for helpful communications.

\section{Notation and basic results} \label{sectionnotation}

We work over the complex field $\C$. 
A fibration is a projective surjective morphism \holom{\varphi}{X}{Y} with connected fibres
between normal varieties such that $\dim X>\dim Y$. 
The $\varphi$-equidimensional (resp. $\varphi$-smooth)  locus is the largest Zariski open subset $Y^* \subset Y$
such that for every $y \in Y^*$, the fibre $\fibre{\varphi}{y}$ has dimension $\dim X - \dim Y$ (resp. has dimension $\dim X - \dim Y$ and is smooth).

An elementary Mori contraction of a quasi-projective manifold $X$ is
a morphism with connected fibres \holom{\varphi}{X}{Y} onto a normal variety $Y$ such that the anticanonical divisor $K_X^*$ is $\varphi$-ample
and the numerical classes of curves contracted by $\varphi$ lie on an extremal ray $R \subset N_1(X)$.

We will use the notation of \cite[Ch.II]{Kol96}:
if $U \rightarrow V$ is a variety $U$ that is projective over some base $V$, we denote by
$\Rat(U/V)$ the space parameterising rational curves on $X$. If moreover $s\colon V \rightarrow U$ is a section,
we denote by $\Rat(s, U/V)$ the space parameterising rational curves passing through $s(V)$.
In the case where $V$ is a point and $x=s(V)$, we simply write $\Rat(U)$ and $\Rat(x, U)$.

Let $F$ be a normal, projective variety of dimension $d$.
If $\Hilb \subset \Rat(F)$ is an irreducible component parameterising a family of rational curves
that dominates $F$, then for a general point $x \in F$ one has
$$
\dim \Hilb_x = \dim \Hilb+1-d,
$$
where $\Hilb_x \subset \Rat(x, F)$ parameterises the members of $\Hilb$ passing through $x$.
If $\Hilb_x$ is proper, it follows from bend-and-break \cite{Mor79} that
$\dim \Hilb_x \leq d-1$.
Thus in this case we have 
\begin{equation} \label{eqnupperbound}
\dim \Hilb \leq 2d-2.
\end{equation}
Fix now an ample $\Q$-Cartier divisor $A$ on $F$. 
Let $C \subset F$ be a rational curve such that 
$$
A \cdot C = \min \{ A \cdot C' \ | \ C' \subset F \ \mbox{is a rational curve} \};
$$ 
then any irreducible component of $\Hilb \subset \Rat(F)$ containing $C$ 
is proper \cite[II.,Prop.2.14]{Kol96}, so it parameterises a family of rational curves $\Hilb$
that is unsplit in the sense of \cite[Defn.0.2]{CMS02}.
If moreover the irreducible component $\Hilb$ has the maximal dimension $2d-2$,
the family is doubly dominant, i.e. for every $x,y \in F$ 
there exists a member of the family $\Hilb$ that joins $x$ and $y$.
In this setting one can prove a version of Theorem \ref{theoremCMS} that includes normal varieties:

\begin{theorem} \cite[Main Thm.0.1]{CMS02}\footnote{A proof of this statement following 
the strategy in \cite{Keb02} can be found in \cite[\S 4]{CT07}.} \label{theoremCMSnormal}
Let $F$ be a normal, projective variety of dimension $d$ and $A$ an ample $\Q$-Cartier divisor on $F$. 
Suppose that there exists an irreducible component $\Hilb \subset \Rat(F)$
of dimension at least $2d-2$ such that for a curve $[C] \in \Hilb$ we have
$$
A \cdot C = \min \{ A \cdot C' \ | \ C' \subset F \ \mbox{is a rational curve} \}.
$$ 
Then $F$ is isomorphic to a projective space $\PP^d$ and $\Hilb$ parameterises the family of lines on $\PP^d$.
\end{theorem}

\begin{remark} \label{remarkbirationalcase} 
If $F$ is not smooth, it is in general quite hard to verify the condition $\dim \Hilb \geq 2d-2$ (cf. \cite{CT07}).
Since we work with an ambient space that is smooth, things are much simpler:
 
Let $X$ be a quasi-projective manifold, and let \holom{\varphi}{X}{Y} 
be a birational elementary contraction. Let $E \subset X$ be an irreducible component of the $\varphi$-exceptional locus
such that for an irreducible component $F$ of a general fibre of $E \rightarrow \varphi(E)$ the inequality \eqref{eqniw} is an equality.
We want to describe the structure of $F$: 
taking general hyperplane sections on $Y$ we can suppose that $E=F$, so the boundary case of \eqref{eqniw} simplifies to
$$
2 \dim F = \dim X + l(R)-1.
$$
Let $C \subset F \subset X$ be a rational curve passing through a general point $x \in F$ that has minimal degree with respect to $K_X^*$.
Then $C$ belongs to an irreducible family of rational curves $\Hilb \subset \Rat(F)$ that dominates $F$, moreover $\Hilb_x$ is proper.
Note now that any deformation of $C$ in $X$ is contracted by $\varphi$, hence contained in $F$.
Thus we can estimate $\dim \Hilb$ by applying Riemann--Roch on the {\em manifold~$X$}: 
$$
\dim \Hilb \geq K_X^* \cdot C-3+\dim X \geq l(R)-3+\dim X=2 \dim F-2.
$$
By \eqref{eqnupperbound} we see that these inequalities are in fact equalities, in particular
one has $K_X^* \cdot C=l(R)$. Thus if $\nu\colon \tilde F \rightarrow F$
is the normalisation and $\tilde \Hilb \subset \Rat(\tilde F)$ the family obtained by lifting the members of $\Hilb$, 
it satisfies the conditions of Theorem~\ref{theoremCMSnormal} with respect to the polarisation $A:=\nu^* K_X^*$.
Thus we have $\tilde F \simeq \PP^{\dim F}$ and $\nu^* K_X^* \equiv l(R) H$ with $H$ the hyperplane divisor. 
\end{remark}

\section{Degenerations of $\PP^d$} \label{sectiondegeneration}

In this section we prove our main results on degenerations of $\PP^d$ satisfying a
length condition. 

\begin{proposition} \label{propositionuniversalfamily}
Let $X$ be a normal, quasi-projective variety 
and let $\holom{\varphi}{X}{Y}$ be an equidimensional fibration
of relative dimension $d$ onto a normal variety such that the general fibre $F$ is isomorphic
to $\PP^d$. Let $A$ be a $\varphi$-ample $\Q$-Cartier divisor, and let $e \in \N$ be such that
$A|_F \equiv eH$ with $H$ the hyperplane divisor.
Suppose that the following length condition holds:
\begin{equation} \label{conditionlength}
A \cdot C \geq e \qquad \forall \ C \subset X \ \mbox{rational curve s.t.} \ \varphi(C)=pt.
\end{equation}
Then all the fibres are irreducible and generically reduced. 
Moreover the normalisation of any fibre is a projective space.
\end{proposition}

\begin{remark}
In general the degeneration behaviour of projective spaces can be quite complicated, 
for example it depends in a subtle manner on the geometry of the total space.
Consider $\holom{\varphi}{X}{C}$ a fibration from a normal variety $X$ onto a smooth curve $C$ such that
all the fibres are integral and the general fibre is isomorphic to $\PP^d$. If $X$ is smooth (or at least factorial),
then Tsen's theorem implies that $X \rightarrow C$ is a $\PP^d$-bundle (\cite[Lemma 2.17]{NO07}, cf. also the proof of \cite[Lemma 4]{DP10}).
This assumption can not be weakened:

\begin{example} \cite{Ara09}
Let $W_4 \subset \PP^6$ be the cone 
over the Veronese surface $R_4 \subset \PP^5$ (that is the $2$-uple embedding of $\PP^2$ in $\PP^5$).
The blow-up of $W_4$ in the vertex $P$ is isomorphic to the projectivised bundle $\PP(\sO_{\PP^2} \oplus \sO_{\PP^2}(-2)) \rightarrow \PP^2$; this desingularisation contracts a $\PP^2$ with normal bundle $\sO_{\PP^2}(-2)$ onto the vertex $P$.
In particular $W_4$ is a terminal $\Q$-factorial threefold and the canonical divisor $K_{W_4}$ is
not Gorenstein, but $2$-Gorenstein.

The base locus of a general pencil of hyperplane sections of $W_4 \subset \PP^6$ identifies
to a smooth quartic curve $C \subset R_4$. If we denote by $\holom{\mu}{X}{W_4}$ the blow-up 
in $C$, it is a terminal, $\Q$-factorial, $2$-Gorenstein threefold admitting a fibration 
$\holom{\varphi}{X}{\PP^1}$ such that the fibres are isomorphic to the members of the general pencil, in particular
they are integral.
Thus the general fibre $F$ is isomorphic to $R_4 \simeq \PP^2$, but the fibre $F_0$ corresponding
to the hyperplane section through the vertex $P$ is a cone over the quartic curve $C$.
In particular it is not normalised by $\PP^2$.

Let us note that $A:=K_X \otimes \mu^* \sO_{W_4}(2)$ is a $\varphi$-ample $\Q$-Cartier divisor
such that the restriction to a general fibre  is numerically equivalent to the hyperplane divisor $H \subset \PP^2$.
However this divisor is not Cartier, so Theorem \ref{theoremfujita} does not apply. 
Let us also note that under the $2$-uple embedding $\PP^2 \hookrightarrow \PP^5$,
a line is mapped onto a conic. If we degenerate the general fibre $F$ to $F_0$, these conics degenerate 
to a union of two lines $l_1 \cup l_2$ passing through the vertex of the cone $F_0$. Since we have $A \cdot l_i = \frac{1}{2}$,
the length condition \eqref{conditionlength} is not satisfied.  
\end{example}
\end{remark}

\begin{proof}[Proof of Proposition \ref{propositionuniversalfamily}]
By assumption the general fibre is a projective space.  
Let $\mathcal H \subset \Rat(X/Y)$ be the 
unique  
irreducible component such that a general point corresponds to a line $l$ contained in the general fibre $F$.
We have $A \cdot l=e$ and by \eqref{conditionlength} one has $A \cdot C \geq e$ for every rational curve $C$ contained in a fibre, so the variety $\mathcal H$ is proper over the base $Y$ \cite[II.,Prop.2.14]{Kol96}.
Since the general fibre of $\Hilb \rightarrow Y$ corresponds to the $2d-2$-dimensional family
of lines in the projective space $\PP^{d} \simeq F$, it follows by upper semicontinuity
that for every $0 \in Y$, all the irreducible components of the fibre $\Hilb_0$ have dimension at least $2d-2$.
Let $\Hilb_{0,i}$ be the normalisation of such an irreducible component and $\Univ_{0,i} \rightarrow \Hilb_{0,i}$ be the universal family over it.
The image of the evaluation morphism $\holom{p}{\Univ_{0,i}}{X}$ is an irreducible component $D_i$
of the set-theoretical fibre $(\fibre{\varphi}{0})_{\red}$, so it has dimension $d$. Let $\nu\colon \tilde D_i \rightarrow D_i$ be the normalisation,
then the family of rational curves  $\Hilb_{0,i}$ lifts to $\tilde D_i$, so $\Rat(\tilde D_i)$ has an irreducible component
of dimension at least $2d-2$.
Since for any rational curve $C \subset D_i \subset X$ we have $A \cdot C \geq e$ by \eqref{conditionlength},
the pull-back $\nu^* A$ also satisfies this inequality.
Moreover $\nu^* A$ has degree exactly $e$ on the rational curves parameterised by $\Hilb_{0,i}$.
Thus we conclude with Theorem \ref{theoremCMSnormal} that $\tilde{D_i}$ is isomorphic to a projective space.

We argue now by contradiction and suppose that there exists a $0 \in Y$ such that the fibre $F_0:=\fibre{\varphi}{0}$ is
reducible or not generically reduced. Then we can decompose the cycle
$$
[F_0] = \sum_i a_i [D_i]
$$
with $a_i \in \N$ and $D_i$ the irreducible components of the set-theoretical fibre $(\fibre{\varphi}{0})_{\red}$.
Since the degree is constant in a well-defined family of proper 
algebraic cycles \cite[Prop.3.12]{Kol96},
we have 
$$
e^{d}= [F] \cdot A^{d}= [F_0] \cdot A^d =  \sum_i a_i ([D_i] \cdot A^{d}).
$$ 
By assumption the sum on the right hand side is not trivial, so we have $[D_1] \cdot A^d<e^d$.
Denote by $\holom{\nu}{\PP^{d}}{D_1}$ the normalisation, then we have
$$
(\nu^* A)^{d}= [D_i] \cdot A^{d} < e^{d}.
$$
Thus we see that $\nu^* A \equiv bH$ with $0<b<e$. Yet this is impossible, since it implies that for a line $l \subset \PP^{d}$, we have
$$
A \cdot \nu(l) = \nu^* A \cdot l = b < e, 
$$
a contradiction to \eqref{conditionlength}. This shows that the cycle-theoretic fibre is smooth in a general point.
By \cite[I, Thm.6.5]{Kol96} this implies that the scheme-theoretic fibre is smooth in such a point. In particular
the scheme-theoretic fibre is generically reduced.
\end{proof}

\begin{lemma} \label{lemmaqcartier}
Let $\holom{\varphi}{X}{Y}$ be an equidimensional fibration from a normal variety $X$ onto a manifold $Y$. Suppose that
$\varphi$ is generically smooth and the general fibre $F$ is a Fano manifold with Picard number one.
If $\varphi$ has irreducible and generically reduced fibres, then $X$ is $\Q$-Gorenstein.
\end{lemma}

\begin{proof} Denote by $H$ the ample generator of $\pic(F)$ and let $A$ be a $\varphi$-ample Cartier divisor.
Set $e \in \N$ such that $A|_F \equiv eH$, and $d \in \N$ such that $K^*_F \equiv dH$.
The reflexive sheaf $\sO_X(e K_X+d A)$
is locally free in a neighborhood of a general fibre $F$ and the restriction
$\sO_X(e K_X+d A) \otimes \sO_F$ is isomorphic to the structure sheaf $\sO_F$. 
Up to replacing $A$ by $A \otimes \varphi^* H$ with $H$ a sufficiently ample Cartier divisor on $Y$ we can suppose without loss of
generality that
$$
H^0(X, \sO_X(e K_X+d A)) \simeq H^0(Y, \varphi_* \sO_X(e K_X+d A)) \neq 0.
$$
Thus we obtain a non-zero morphism
$$
\sO_X(eK_X^*) \rightarrow \sO_X(d A)
$$
which is an isomorphism on the general fibre $F$. The vanishing locus is thus a Weil divisor
$\sum k_i D_i$ such that for all $i$ we have $\varphi(D_i) \subsetneq X$.
Since $\varphi$ is equidimensional and all the fibres are irreducible and generically reduced
we see that $D_i = \varphi^* E_i$ with $E_i$ a prime divisor on $Y$. Since $Y$ is smooth, the divisor $E_i$
is Cartier, hence $\sum k_i D_i=\sum k_i \varphi^* E_i$ is Cartier. Thus we have an isomorphism
$$
\sO_X(eK_X^*) \simeq \sO_X(d A-\sum k_i \varphi^* E_i).
$$
The right hand side is Cartier, so $K_X$ is $\Q$-Cartier. 
\end{proof}

\begin{proposition} \label{propositionbundle}
In the situation of Proposition \ref{propositionuniversalfamily}, suppose that $Y$ is smooth.
Then the fibration $\varphi$ is a projective bundle.
\end{proposition}

\begin{proof}
We will proceed by induction on the relative dimension $D$, the case $d=1$ is \cite[II,Thm.2.8]{Kol96}.
Taking general hyperplane sections on $Y$ and arguing by induction on the dimension we can suppose without
loss of generality that $\varphi$ has at most finitely many singular fibres. The problem being local
we can suppose that $Y \subset \C^{\dim Y}$ is a polydisc around $0$ and $\varphi$ is smooth
over $Y \setminus 0$. Since all the fibres are generically reduced by Proposition \ref{propositionuniversalfamily}
we can choose a section $s\colon Y \rightarrow X$ such that $s(0) \subset F_{0, \nons}$ where $F_0:=\fibre{\varphi}{0}$.

As in the proof of Proposition \ref{propositionuniversalfamily} we denote by $\mathcal H \subset \Rat(X/Y)$ the 
unique irreducible component such that a general point corresponds to a line $l$ contained in the general fibre $F \simeq \PP^d$. 
There exists a unique irreducible component $\mathcal H_s \subset \Rat(s, X/Y)$ 
such that for general $y \in Y$ the lines in $\fibre{\varphi}{y} \simeq \PP^d$ passing through $s(y)$ are parameterised by $\mathcal H_s$.
If we denote by $\holom{\psi}{\mathcal H_s}{Y}$
the natural fibration, its general fibre is isomorphic to a projective space $\PP^{d-1}$.
Let $\holom{q}{\mathcal U_s}{\mathcal H_s}$ be the universal family, then by \cite[Ch.II, Cor.2.12]{Kol96}
the fibration $q$ is a $\PP^1$-bundle. Let $\holom{p}{\mathcal U_s}{X}$ be the evaluation morphism,
then $p$ is birational since this holds for the restriction to a general $\varphi$-fibre.
The variety $X$ being normal, we know by Zariski's main theorem that $p$ has connected fibres.
The family of rational curves being unsplit, it follows from bend-and-break that $p$ has finite fibres over
$X \setminus s(Y)$. Thus we have an isomorphism
\begin{equation} \label{theiso}
X \setminus s(Y) \simeq \mathcal U_s \setminus E,
\end{equation}
where $E:=(\fibre{p}{s(Y)})_{\red}$. 
Again by bend-and-break and connectedness of the fibres the algebraic set $E$ is irreducible, so it is a prime divisor with a finite,
birational morphism $q|_E\colon E \rightarrow \mathcal H_s$. Since $\mathcal H_s$ is normal, we see
by Zariski's main theorem that $q|_E$ is an isomorphism and $E$ is a $q$-section.

By Lemma \ref{lemmaqcartier} we know that $X$ is $\Q$-Gorenstein, so $\mathcal U_s \setminus E$
is $\Q$-Gorenstein.
Since $(\mathcal U_s \setminus E) \rightarrow \mathcal H_s$ is locally trivial (it is a $\C$-bundle),
it follows that $\mathcal H_s$ is $\Q$-Gorenstein. Since $\mathcal U_s \rightarrow \mathcal H_s$
is a $\PP^1$-bundle, it now follows that the total space $\mathcal U_s$ is $\Q$-Gorenstein.
The $\PP^1$-bundle $\mathcal U_s \rightarrow \mathcal H_s$ has a section $E$, so it is isomorphic to
the projectivised bundle $\PP(V) \rightarrow \mathcal H_s$ with $V:=\varphi_* \sO_X(E)$. 
By the canonical bundle formula we have
\begin{equation} \label{canonicalbundleformula}
K_{\mathcal U_s} \equiv q^*(K_{\mathcal H_s}+\det V) - 2 E.
\end{equation}
We claim that the following length condition holds: 
\begin{equation} \label{checkcondition}
K^*_{\mathcal H_s} \cdot C \geq d \qquad \forall \ C \subset \mathcal H_s \ \mbox{rational curve s.t.} \ \psi(C)=pt.
\end{equation}
Assuming this for the time being, let us see how to conclude: applying the induction hypothesis
to $\psi$, we see that $\mathcal H_s \rightarrow Y$ is a $\PP^{d-1}$-bundle. Since $q$ is a $\PP^1$-bundle,
we see that the central fibre of $\mathcal U_s \stackrel{\psi \circ q}{\rightarrow} Y$ is a $\PP^1$-bundle over $\PP^{d-1}$ that contracts a section onto a smooth point of $F_0$.
A classical argument \cite{Mor79} shows that $F_0 \simeq \PP^d$.

{\em Proof of the claim.}
The claim is obvious for curves in the general fibres which are isomorphic to $\PP^{d-1}$, so we can concentrate
on curves in the central fibre $\mathcal H_0 := (\fibre{\psi}{0})_{\red}$. We set
$\mathcal U_0 := \fibre{q}{\mathcal H_0}$ and denote by $\holom{\nu}{\tilde{\mathcal H}_0}{\mathcal H_0}$
the normalisation. Then $\mathcal U_0 \times_{\mathcal H_0} \tilde{\mathcal H}_0$
is normal and $\holom{\tilde q_0}{\mathcal U_0 \times_{\mathcal H_0} \tilde{\mathcal H}_0}{\tilde{\mathcal H}_0}$
is a $\PP^1$-bundle. By Proposition \ref{propositionuniversalfamily} we already know
that $F_{0, \red}$ is normalised by $\PP^d$, so we obtain a commutative diagram:
$$
 \xymatrix{ 
& & \PP^d \ar[d]
\\
\mathcal U_0 \times_{\mathcal H_0} \tilde{\mathcal H}_0
\ar[r]_{\tilde \nu} \ar[d]_{\tilde q} \ar[rru]^\mu
&
\mathcal U_0 \ar[r]^{p_0} \ar[d]^{q_0} 
&
F_{0, \red}
\\
\tilde{\mathcal H}_0 \ar[r]_\nu
&
\mathcal H_0
& 
} 
$$
Thus the curves parameterised by $\tilde{\mathcal H}_0$ are lines in $\PP^d$, in particular we have $\tilde{\mathcal H}_0 \simeq \PP^{d-1}$
and 
$$
\mathcal U_0 \times_{\mathcal H_0} \tilde{\mathcal H}_0
\simeq
\PP(\sO_{\PP^{d-1}} \oplus \sO_{\PP^{d-1}}(-1))
$$
with $\fibre{\tilde \nu}{E \cap \mathcal U_0}$ corresponding to the exceptional section.

In order to check the length condition \eqref{checkcondition} it is sufficient to do this for a curve $C_0 \subset \mathcal H_0$
such that $C_0=\nu(l_0)$ with $l_0$ a general line in $\tilde{\mathcal H}_0 \simeq \PP^{d-1}$.
We can lift $l_0$ to a curve $l' \subset \PP(\sO_{\PP^{d-1}} \oplus \sO_{\PP^{d-1}}(-1))$
that is disjoint from the exceptional section, so we obtain a curve $C':=\tilde \nu(l') \subset \mathcal U_0 \setminus E$
such that $C' \rightarrow C_0$ is birational. In particular we have
$$
K^*_{\mathcal H_s} \cdot C_0 =  q^* K^*_{\mathcal H_s} \cdot C', 
$$
so it is sufficient to show that the right hand side is equal to $d$.  The rational curve $C'$ does not meet 
$E$ and $\mu(l')$ is a line, so $C'$ corresponds to a point in  $\mathcal H$ that is not in $\mathcal H_s$. Using the isomorphism
\eqref{theiso} we can deform $C'$ to a curve $C'' \subset \mathcal U_s$ that is
a line in a general fibre $\fibre{\varphi}{y} \simeq \PP^d$ and which does not meet the point $s(y)$,
so $C''$ is disjoint from $E$. 
Yet for such a curve the Formula \eqref{canonicalbundleformula}
restricted to a general fibre immediately shows that $q^* K^*_{\mathcal H_s} \cdot C''=d$.
\end{proof}

\begin{remark} \label{remarknormal}
Proposition \ref{propositionbundle} should be true without the assumption that $Y$ is smooth.
In fact this assumption is only needed to assure via Lemma \ref{lemmaqcartier} 
that $X$ (and by consequence $\mathcal H_s$ and $\mathcal U_s$) are $\Q$-Gorenstein.
However our computations only use that these varieties are ``relatively  $\Q$-Gorenstein'', i.e.
some multiple of the canonical divisor is linearly equivalent to a Cartier divisor plus some Weil divisors
that are pull-backs from the base $Y$. We leave the technical details to the interested reader.
\end{remark}

\section{Proofs of main results}

\begin{proof}[Proof of Theorem \ref{theoremmain}]
By Theorem \ref{theoremCMS} the general $\varphi$-fibre is a projective space.

If $\varphi$ is equidimensional, then $A:=K_X^*$ satisfies the length 
condition \eqref{conditionlength} in Proposition \ref{propositionuniversalfamily} with $e=l(R)$. 
Thus we know that every fibre is generically reduced.
In particular for every $y \in Y$ there exists a point $x \in \fibre{\varphi}{y}$ such that the tangent map
has rank $\dim Y$ in $x$. Since $X$ is smooth, this implies that $Y$ is smooth in $y$. 
Conclude with Proposition \ref{propositionbundle}.

Suppose now that $\varphi$ is not equidimensional.
Let now $\bar Y$ be the closure of the $\varphi$-equidimensional locus in the Chow scheme $\chow{X}$ and let $\bar X \rightarrow \bar Y$ be the universal family. 
Let $Y' \rightarrow \bar Y$ be a desingularisation, and set $X'$ for the normalisation of $\bar X \times_{\bar Y} Y'$.
We denote by \holom{\varphi'}{X'}{Y'} the natural fibre space structure and by \holom{\mu'}{X'}{X}
the birational morphism induced by the map $\bar X \rightarrow X$.
By the rigidity lemma there exists a birational morphism \holom{\mu}{Y'}{Y}
such that $\mu \circ \varphi'=\varphi \circ \mu'$. Note also that the restriction of $\mu'$ to any $\varphi'$-fibre
is finite, so the pull-back $A:=(\mu')^* K_X^*$ is a $\varphi'$-ample Cartier divisor.
The divisor $A$ on $X'$ satisfies the length condition \eqref{conditionlength}
in Proposition \ref{propositionuniversalfamily} with $e=l(R)$. 
Since $Y'$ is smooth, we can conclude with the Proposition \ref{propositionbundle}.
\end{proof}

\begin{proof}[Proof of Theorem \ref{theorembirational}]
By Remark \ref{remarkbirationalcase} the general $\tilde \varphi$-fibre is a projective space.
Moreover if $\holom{\nu}{\tilde E}{E \subset X}$ denotes the normalisation, then
$A:=\nu^* K_X^*$ (restricted to the $\tilde \varphi$-equidimensional locus) satisfies the length 
condition \eqref{conditionlength} in Proposition \ref{propositionuniversalfamily} with $e=l(R)$. 
If $C \subset \tilde Z$ is a curve cut out by general hyperplane sections, then $C$ is smooth and  
the fibration $\fibre{\tilde \varphi}{C} \rightarrow C$ is equidimensional, so it is a projective bundle by Proposition \ref{propositionbundle}. Thus $\tilde \varphi$ is a projective bundle in codimension one. 

The proof of the second statement is analogous to the proof of the second statement in Theorem \ref{theoremmain}.
\end{proof}

\section{Local scroll structures} \label{sectionscroll}
We will use the following local version of \cite[Defn.3.3.1]{BS95}:

\begin{definition}
Let $X$ be a manifold, and let \holom{\varphi}{X}{Y} be a fibration onto a normal variety $Y$.
We say that $\varphi$ is locally a scroll if for every $y \in Y$ there exists an analytic neighborhood
$y \in U \subset Y$ such that on $X_U:=\fibre{\varphi}{U}$ there exists a relatively ample Cartier divisor
$L$ such that the general fibre $F$  polarised by $L|_F$ is isomorphic to $(\PP^d, H)$.
\end{definition}

Suppose now that we are in the situation of Theorem \ref{theoremmain} and denote by $Y_0 \subset Y$ the $\varphi$-equidimensional locus.
Then $Y_0$ is also the $\varphi$-smooth locus by Theorem \ref{theoremmain}. 
Note also that $Y$ is factorial and has at most rational singularities \cite{Kol86},
so it has at most canonical singularities \cite[Cor.5.24]{KM98}.
Thus if $\mu\colon Y' \rightarrow Y$ is the resolution of singularities in the statement of the theorem,
we know by \cite{Tak03} that for any point $y \in Y$ there exists a contractible analytic neighborhood $U \subset Y$ such that
$\fibre{\mu}{U}$ is simply connected. 
Thus the local systems $\varphi'_* \Z_{X'}$ and $R^2 \varphi'_* \Z_{X'}$ are trivial on $\fibre{\mu}{U}$.
Since the restriction of these local systems to $U_0:= Y_0 \cap U$ coincides with 
$(\varphi_* \Z_{X})_{U_0}$ and $(R^2 \varphi_* \Z_{X})_{U_0}$, the latter are trivial on $U_0$.

\begin{lemma} \label{lemmalocaloone}
In the situation of Theorem \ref{theoremmain}, using the notation above, suppose that $H^3(U_0, \Z)$ is torsion-free.
Then $\varphi$ is a scroll over $U$.
\end{lemma}

\begin{proof} We follow the argument in \cite[Lemma 3.3]{AM97}, \cite{Ara09}.
Set $X_0:=\fibre{\varphi}{U_0}$. 
We use the Leray spectral sequence for the smooth morphism $\holom{\varphi}{X_0}{U_0}$,
so set $E^{p,q}_2=H^p(U_0, R^q \varphi_* \Z_{X_0})$. We have $R^1 \varphi_* \Z_{X_0}=0$, 
so $E^{p,1}_2=0$ for all $p$. By what precedes we have $\varphi_* \Z_{X_0} \simeq \Z_{U_0}$ and $R^2 \varphi_* \Z_{X_0} \simeq \Z_{U_0}$. 
Since the fibre of $R^2 \varphi_* \Z_{X_0}$ in any point $y \in U_0$ identifies to
$H^2(X_y, \Z) \simeq \pic(\PP^d) \simeq \Z [H]$, we have
$$
E^{0,2}_2 \simeq \Z [H].
$$
The spectral sequence gives a map $d_3\colon E^{0,2}_3 \rightarrow E^{3,0}_3$ with kernel  
$E^{0,2}_4=E^{0,2}_\infty$. We have
$$
E^{0,2}_3  \simeq E^{0,2}_2 \simeq \Z [H]
$$
and
$$
E^{3,0}_3 \simeq E^{3,0}_2 \simeq H^3(U_0, \Z),
$$
so we get an exact sequence
$$
0 \rightarrow E^{0,2}_\infty \rightarrow E^{0,2}_2 \rightarrow H^3(U_0, \Z).
$$
Consider now the composed map
$$
r\colon \pic(X_0) \rightarrow H^2(X_0, \Z) \rightarrow E^{0,2}_\infty \rightarrow E^{0,2}_2 \simeq \Z [H].
$$
By the exact sequence above we know that the cokernel of $r$ injects into $H^3(U_0, \Z)$.
Since $\varphi$ is projective, the restriction map $r$ is not zero, so its cokernel is torsion. By hypothesis $H^3(U_0, \Z)$
is torsion-free, so $r$ is surjective.

Thus there exists a Cartier divisor $L_0$ on $X_0$ such that $L_0|_F \equiv H$.
Since $\varphi$ is smooth in codimension one and 
does not contract a divisor, the set $X_U \setminus X_0$ has codimension at least two in $X_U:=\fibre{\varphi}{U}$.
Therefore $L_0$ extends to a $\varphi$-ample Cartier divisor $L$ on $X_U$.
\end{proof}

\begin{remark} \label{remarklocaloone}
Let us note that if $\pi_1(U_0) \simeq \pi_2(U_0) \simeq \{1\}$ the technical condition in the preceding lemma is satisfied:
by the Hurewicz theorem \cite[Thm.4.32]{Hat02} one has $H_2(U_0, \Z)=0$, so $H^3(U_0, \Z)$ is torsion-free by \cite[Cor.3.3]{Hat02}.
The topological conditions $\pi_1(U_0) \simeq \pi_2(U_0)=\{1\}$ are known to be satisfied in the following two situations:
\begin{itemize}
\item $U$ is smooth\footnote{Note that $U \setminus U_0$ has codimension at least three, so
$\pi_i(U) \simeq \pi_i(U_0)$ for $i=1,2$.}; or
\item $\dim U \geq 4$ and $U$ has at most isolated lci singularities \cite[Kor.1.3]{Ham71}.
\end{itemize}
If $\dim U=3$ Hamm's theorem does not apply to $\pi_2(U_0)$ and it seems not to be clear if
$H^3(U_0, \Z)$ is torsion-free, even for cDV-singularities.
\end{remark}


\begin{thebibliography}{CMSB02}

\bibitem[ABW92]{ABW92}
Marco Andreatta, Edoardo Ballico, and Jaros{\l}aw~A. Wi{\'s}niewski.
\newblock Projective manifolds containing large linear subspaces.
\newblock In {\em Classification of irregular varieties ({T}rento, 1990)},
  volume 1515 of {\em Lecture Notes in Math.}, pages 1--11. Springer, Berlin,
  1992.

\bibitem[AM97]{AM97}
Marco Andreatta and Massimiliano Mella.
\newblock Contractions on a manifold polarized by an ample vector bundle.
\newblock {\em Trans. Amer. Math. Soc.}, 349(11):4669--4683, 1997.

\bibitem[AO02]{AO02}
Marco Andreatta and Gianluca Occhetta.
\newblock Special rays in the {M}ori cone of a projective variety.
\newblock {\em Nagoya Math. J.}, 168:127--137, 2002.

\bibitem[Ara09]{Ara09}
Carolina Araujo.
\newblock Flat deformations of {$\mathbb P^n$}.
\newblock {\em unpublished preprint}, 2009.

\bibitem[AW97]{AW97}
Marco Andreatta and Jaros{\l}aw~A. Wi{\'s}niewski.
\newblock A view on contractions of higher-dimensional varieties.
\newblock In {\em Algebraic geometry---{S}anta {C}ruz 1995}, volume~62 of {\em
  Proc. Sympos. Pure Math.}, pages 153--183. Amer. Math. Soc., Providence, RI,
  1997.

\bibitem[BS93]{BS93}
Mauro~C. Beltrametti and Andrew~J. Sommese.
\newblock Comparing the classical and the adjunction-theoretic definition of
  scrolls.
\newblock In {\em Geometry of complex projective varieties ({C}etraro, 1990)},
  volume~9 of {\em Sem. Conf.}, pages 55--74. Mediterranean, Rende, 1993.

\bibitem[BS95]{BS95}
Mauro~C. Beltrametti and Andrew~J. Sommese.
\newblock {\em The adjunction theory of complex projective varieties},
  volume~16 of {\em de Gruyter Expositions in Mathematics}.
\newblock Walter de Gruyter \& Co., Berlin, 1995.

\bibitem[BSW92]{BSW92}
Mauro~C. Beltrametti, Andrew~J. Sommese, and Jaros{\l}aw~A. Wi{\'s}niewski.
\newblock Results on varieties with many lines and their applications to
  adjunction theory.
\newblock In {\em Complex algebraic varieties ({B}ayreuth, 1990)}, volume 1507
  of {\em Lecture Notes in Math.}, pages 16--38. Springer, Berlin, 1992.

\bibitem[CMSB02]{CMS02}
Koji Cho, Yoichi Miyaoka, and Nicholas~I. Shepherd-Barron.
\newblock Characterizations of projective space and applications to complex
  symplectic manifolds.
\newblock In {\em Higher dimensional birational geometry ({K}yoto, 1997)},
  volume~35 of {\em Adv. Stud. Pure Math.}, pages 1--88. Math. Soc. Japan,
  Tokyo, 2002.

\bibitem[CT07]{CT07}
Jiun-Cheng Chen and Hsian-Hua Tseng.
\newblock Note on characterization of projective spaces.
\newblock {\em Comm. Algebra}, 35(11):3808--3819, 2007.

\bibitem[DP10]{DP10}
St{\'e}phane Druel and Matthieu Paris.
\newblock Characterizations of projective spaces and hyperquadrics.
\newblock {\em arxiv preprint}, 1012.5238, 2010.

\bibitem[Ein85]{Ein85}
Lawrence Ein.
\newblock Varieties with small dual varieties. {II}.
\newblock {\em Duke Math. J.}, 52(4):895--907, 1985.

\bibitem[Fuj87]{Fuj87}
Takao Fujita.
\newblock On polarized manifolds whose adjoint bundles are not semipositive.
\newblock In {\em Algebraic geometry, {S}endai, 1985}, volume~10 of {\em Adv.
  Stud. Pure Math.}, pages 167--178. North-Holland, Amsterdam, 1987.

\bibitem[Ham71]{Ham71}
Helmut Hamm.
\newblock Lokale topologische {E}igenschaften komplexer {R}\"aume.
\newblock {\em Math. Ann.}, 191:235--252, 1971.

\bibitem[Hat02]{Hat02}
Allen Hatcher.
\newblock {\em Algebraic topology}.
\newblock Cambridge University Press, Cambridge, 2002.

\bibitem[Ion86]{Ion86}
Paltin Ionescu.
\newblock Generalized adjunction and applications.
\newblock {\em Math. Proc. Cambridge Philos. Soc.}, 99(3):457--472, 1986.

\bibitem[Kaw89]{Ka89}
Yujiro Kawamata.
\newblock Small contractions of four-dimensional algebraic manifolds.
\newblock {\em Math. Ann.}, 284(4):595--600, 1989.

\bibitem[Keb02]{Keb02}
Stefan Kebekus.
\newblock Characterizing the projective space after {C}ho, {M}iyaoka and
  {S}hepherd-{B}arron.
\newblock In {\em Complex geometry ({G}\"ottingen, 2000)}, pages 147--155.
  Springer, Berlin, 2002.

\bibitem[KM98]{KM98}
J{\'a}nos Koll{\'a}r and Shigefumi Mori.
\newblock {\em Birational geometry of algebraic varieties}, volume 134 of {\em
  Cambridge Tracts in Mathematics}.
\newblock Cambridge University Press, Cambridge, 1998.
\newblock With the collaboration of C. H. Clemens and A. Corti.

\bibitem[KO73]{KO73}
Shoshichi Kobayashi and Takushiro Ochiai.
\newblock Characterizations of complex projective spaces and hyperquadrics.
\newblock {\em J. Math. Kyoto Univ.}, 13:31--47, 1973.

\bibitem[Kol86]{Kol86}
J{\'a}nos Koll{\'a}r.
\newblock Higher direct images of dualizing sheaves. {I}.
\newblock {\em Ann. of Math. (2)}, 123(1):11--42, 1986.

\bibitem[Kol96]{Kol96}
J{\'a}nos Koll{\'a}r.
\newblock {\em Rational curves on algebraic varieties}, volume~32 of {\em
  Ergebnisse der Mathematik und ihrer Grenzgebiete. 3. Folge. A Series of
  Modern Surveys in Mathematics}.
\newblock Springer-Verlag, Berlin, 1996.

\bibitem[Mor79]{Mor79}
Shigefumi Mori.
\newblock Projective manifolds with ample tangent bundles.
\newblock {\em Ann. of Math. (2)}, 110(3):593--606, 1979.

\bibitem[NO07]{NO07}
Carla Novelli and Gianluca Occhetta.
\newblock Ruled {F}ano fivefolds of index two.
\newblock {\em Indiana Univ. Math. J.}, 56(1):207--241, 2007.

\bibitem[NO11]{NO11}
Carla Novelli and Gianluca Occhetta.
\newblock Projective manifolds containing a large linear subspace with nef
  normal bundle.
\newblock {\em Michigan Math. J.}, 60:441--462, 2011.

\bibitem[Sat97]{Sa97}
Eiichi Sato.
\newblock Projective manifolds swept out by large-dimensional linear spaces.
\newblock {\em Tohoku Math. J. (2)}, 49(3):299--321, 1997.

\bibitem[Som86]{Somm85}
Andrew~J. Sommese.
\newblock On the adjunction theoretic structure of projective varieties.
\newblock In {\em Complex analysis and algebraic geometry ({G}\"ottingen,
  1985)}, volume 1194 of {\em Lecture Notes in Math.}, pages 175--213.
  Springer, Berlin, 1986.

\bibitem[Tak03]{Tak03}
Shigeharu Takayama.
\newblock Local simple connectedness of resolutions of log-terminal
  singularities.
\newblock {\em Internat. J. Math.}, 14(8):825--836, 2003.

\bibitem[Tir10]{Ti10}
Andrea~L. Tironi.
\newblock Scrolls over four dimensional varieties.
\newblock {\em Adv. Geom.}, 10(1):145--159, 2010.

\bibitem[Wi{\'s}91a]{Wis91}
Jaros{\l}aw~A. Wi{\'s}niewski.
\newblock On contractions of extremal rays of {F}ano manifolds.
\newblock {\em J. Reine Angew. Math.}, 417:141--157, 1991.

\bibitem[Wi{\'s}91b]{Wis91b}
Jaros{\l}aw~A. Wi{\'s}niewski.
\newblock On deformation of nef values.
\newblock {\em Duke Math. J.}, 64(2):325--332, 1991.

\end{thebibliography}
\end{document}